\documentclass[11pt]{article}

\usepackage[margin=1.12in]{geometry}
\usepackage[T1]{fontenc}
\usepackage[utf8]{inputenc}
\usepackage{lmodern}
\usepackage{microtype}
\usepackage{amsmath,amssymb,amsthm,mathtools}
\usepackage{enumitem}
\usepackage{float}
\usepackage{xcolor}
\usepackage{tikz}
\usetikzlibrary{positioning,calc}
\usepackage{hyperref}

\hypersetup{
  colorlinks=true,
  linkcolor=blue!45!black,
  citecolor=blue!45!black,
  urlcolor=blue!45!black,
  pdftitle={Odd-Cycle Span Defect: A Polynomial Lower Bound and a Square-Root Upper Bound},
  pdfauthor={Shuyan Chen}
}

\title{Odd-Cycle Span Defect:\\A Polynomial Lower Bound and a Square-Root Upper Bound}
\author{Shuyan Chen\thanks{Department of Mathematics, The University of Manchester, Alan Turing Building, Oxford Road, Manchester M13 9PL, United Kingdom. Email: \href{mailto:shuyan.chen-2@student.manchester.ac.uk}{\texttt{shuyan.chen-2@student.manchester.ac.uk}}.}}
\date{June 2026}

\newtheorem{theorem}{Theorem}[section]
\newtheorem{lemma}[theorem]{Lemma}
\newtheorem{proposition}[theorem]{Proposition}
\newtheorem{corollary}[theorem]{Corollary}
\newtheorem{observation}[theorem]{Observation}

\theoremstyle{definition}

\theoremstyle{remark}
\newtheorem{remark}[theorem]{Remark}
\newtheorem{question}[theorem]{Question}

\newcommand{\ohdef}{\operatorname{ohdef}}
\newcommand{\pathspan}{\pi}
\newcommand{\distH}{d_{\mathrm H}}

\begin{document}

\maketitle

\begin{abstract}
For a graph $G$, let
\[
  \psi(G)=\max\{\chi(G[V(C)]):C\text{ is an odd cycle of }G\},
\]
with $\psi(G)=0$ when $G$ is bipartite.  For positive integers $N$, set
\[
  F(N)=\max\{\chi(G)-\psi(G):|V(G)|\le N\}.
\]
The function $F$ measures the finite-order additive gap arising from an open
problem of Erd\H{o}s and Hajnal.  We prove
\[
  N^{1/6-o(1)}\le F(N)<\sqrt{6N}.
\]
The lower bound raises the finite-order scale supplied by the Cameron--Clow
path-colour construction from $\log N/\log\log N$ to a fixed power of $N$.
Its proof constructs a palette--code graph from a binary covering code
$\mathcal C\subseteq\{0,1\}^p$ and establishes the exact identities
\[
  \chi(G)=2p+\ell-\rho(\mathcal C),
  \qquad
  \psi(G)=2p.
\]
Near-middle Hamming coverings yield the exponent $1/6$.  The upper bound
combines Polavarapu's connectivity theorem, the Chv\'atal--Erd\H{o}s
Hamiltonicity theorem, and maximum-independent-set stripping.
\end{abstract}


\section{Introduction}

A problem of Erd\H{o}s and Hajnal, recorded by Gy\'arf\'as~\cite{Gyarfas1997},
asks whether for every integer $k\ge3$ there is a number $f(k)$ such that every
graph $G$ with $\chi(G)\ge f(k)$ contains an odd cycle $C$ with
\[
  \chi(G[V(C)])\ge k.
\]
The question belongs to the broader programme of $\chi$-boundedness initiated
by Gy\'arf\'as~\cite{Gyarfas1987}; see Scott and Seymour~\cite{ScottSeymour2020}
for a survey.  It remains open and is listed as Problem~640 in Bloom's
catalogue of Erd\H{o}s problems~\cite{Bloom640}.

For a graph $G$, define its \emph{odd-cycle span} by
\[
  \psi(G)=\max\{\chi(G[V(C)]):C\text{ is an odd cycle of }G\},
\]
with $\psi(G)=0$ when $G$ is bipartite.  The cycle $C$ is not required to be
induced: $G[V(C)]$ retains all chords among the vertices of $C$.  We introduce
the finite-order additive extremal function
\[
  \ohdef(G)=\chi(G)-\psi(G),
  \qquad
  F(N)=\max\{\ohdef(G):|V(G)|\le N\},
  \qquad N\in\mathbb N.
\]
Throughout, all graphs are finite and simple, all logarithms are natural, and
$\chi(\varnothing)=\alpha(\varnothing)=\psi(\varnothing)=0$.

The related path parameter
\[
  \pathspan(G)=\max\{\chi(G[V(P)]):P\text{ is a path of }G\}
\]
satisfies $\psi(G)\le\pathspan(G)$.  This inequality transfers finite-order
path constructions to $F$.

The path-colour problem was studied by Randerath and
Schiermeyer~\cite{RanderathSchiermeyer2002} and, more recently, by Cameron and
Clow~\cite{CameronClow2025}.  Inspecting the order growth in the Cameron--Clow recursion gives the
following finite-order consequence.

\begin{proposition}[Cameron--Clow transfer]\label{prop:cc-transfer}
As $N\to\infty$,
\[
  F(N)\ge
  \left(\frac12-o(1)\right)\frac{\log N}{\log\log N}.
\]
\end{proposition}

The leading constant in this estimate comes from the exact stage count in
the Cameron--Clow recursion; Section~\ref{sec:cc-transfer} gives the
calculation.

The palette--code construction gives the following polynomial lower bound.

\begin{theorem}[Polynomial lower bound]\label{thm:lower-main}
For every $\delta>0$ and all sufficiently large $N$,
\[
  F(N)\ge N^{1/6-\delta}.
\]
Equivalently, $F(N)\ge N^{1/6-o(1)}$.
\end{theorem}

The theorem changes the lower-bound scale in Proposition~\ref{prop:cc-transfer}
from $\log N/\log\log N$ to a fixed power of $N$.

We also prove a square-root upper bound.

\begin{theorem}[Square-root upper bound]\label{thm:upper-main}
If $G$ has $N$ vertices and $d=\chi(G)-\psi(G)$, then
\[
  d(d+1)\le6N.
\]
Consequently,
\[
  F(N)<\sqrt{6N}.
\]
\end{theorem}

Together, Theorems~\ref{thm:lower-main} and~\ref{thm:upper-main} give
\[
  \frac16\le
  \liminf_{N\to\infty}\frac{\log F(N)}{\log N}
  \le
  \limsup_{N\to\infty}\frac{\log F(N)}{\log N}
  \le\frac12.
\]

The lower bound rests on an exact coding-theoretic conversion.  From a binary
code $\mathcal C\subseteq\{0,1\}^p$ and an integer $\ell$, we construct a graph
$G(p,\mathcal C,\ell)$ satisfying, under two explicit numerical conditions,
\[
  \chi(G(p,\mathcal C,\ell))=2p+\ell-\rho(\mathcal C),
  \qquad
  \psi(G(p,\mathcal C,\ell))=2p.
\]
The exact defect is $\ell-\rho(\mathcal C)$.  The covering radius controls
the global chromatic number.  Each odd cycle visits only linearly many attached
cliques, which admit a common recolouring inside the original palette.
Near-middle Hamming coverings yield the exponent $1/6$.

For the upper bound, Polavarapu's theorem on highly connected
subgraphs~\cite{Polavarapu2026} and the Chv\'atal--Erd\H{o}s Hamiltonicity
theorem~\cite{ChvatalErdos1972} yield
$\ohdef(G)\le3\alpha(G)$.  Maximum-independent-set stripping proves
Theorem~\ref{thm:upper-main}.  Nguyen's earlier connectivity bound
\cite{Nguyen2024} gives the coefficient $49/8$ by the same argument.

Section~\ref{sec:cc-transfer} proves Proposition~\ref{prop:cc-transfer}.
Section~\ref{sec:upper} proves the upper bound.  Section~\ref{sec:construction}
introduces the palette--code graph and proves the exact conversion.  Section~
\ref{sec:asymptotic} derives the polynomial lower bound.  The final section gives an exponent-transfer principle and a discrepancy
lower bound for near-critical codes.

\section{The Cameron--Clow transfer}\label{sec:cc-transfer}

The proof below extracts the graph order from the recursive construction in
Cameron and Clow~\cite[Sections~4--6]{CameronClow2025}.

\begin{proof}
Fix $r\ge6$ and put
\[
  m=\left\lfloor\frac r2\right\rfloor-1,
  \qquad
  n_k=|V(G_k^{(r)})|,
  \qquad
  s_k=r+k-1.
\]
The construction begins with $G_0^{(r)}=K_r$, so $n_0=r$.
At stage $k+1$, it first forms a graph $H_{k+1}^{(r)}$ from
$\binom{s_k}{2}$ disjoint copies of $G_k^{(r)}$, an independent set of
$s_k$ new vertices, and one further vertex.  Hence
\[
  |V(H_{k+1}^{(r)})|
  =\binom{s_k}{2}n_k+s_k+1.
\]
The construction then takes two extensions of $H_{k+1}^{(r)}$, one with two
new vertices and one with a single new vertex, and identifies two pairs of
vertices between them.  It follows that
\[
  n_{k+1}
  =2\binom{s_k}{2}n_k+2s_k+3.
\]
For $0\le k<m$, we have $s_k\le r+m-2\le3r/2$.  Since $n_k\ge r$, the last
display gives
\[
  n_{k+1}\le Cr^2n_k
\]
for an absolute constant $C$.  Iterating through the $m$ stages yields
\[
  n_m\le r(Cr^2)^m.
\]
Taking logarithms gives
\[
  \log n_m
  \le \log r+m(2\log r+\log C)
  =r\log r+O(r).
\]
For the terminal graph $H_r=G_m^{(r)}$, Cameron and Clow prove
\[
  \pathspan(H_r)\le r,
  \qquad
  \chi(H_r)\ge\left\lfloor\frac{3r}{2}\right\rfloor-1.
\]
Since $\psi(H_r)\le\pathspan(H_r)$,
\[
  \chi(H_r)-\psi(H_r)
  \ge \chi(H_r)-\pathspan(H_r)
  \ge \left\lfloor\frac r2\right\rfloor-1.
\]

For the inversion, let $N$ be large and write
\[
  L_1=\log N,\qquad L_2=\log\log N,\qquad
  L_3=\log\log\log N
\]
and choose
\[
  r=\left\lfloor\frac{L_1}{L_2+3L_3}\right\rfloor.
\]
Then $r=(1-o(1))L_1/L_2$ and $\log r\le L_2-L_3$ for all sufficiently large
$N$.  Hence
\[
  r\log r
  \le L_1\frac{L_2-L_3}{L_2+3L_3}
  =L_1-\frac{4L_1L_3}{L_2+3L_3}.
\]
The error term $O(r)=O(L_1/L_2)$ is smaller than the displayed margin, so
$\log n_m\le L_1$ for all sufficiently large $N$.  Hence $|V(H_r)|\le N$, and
\[
  F(N)\ge\left(\frac12-o(1)\right)
  \frac{\log N}{\log\log N}.
\]
\end{proof}

\section{The square-root upper bound}\label{sec:upper}

We write $\alpha(G)$ for the independence number of $G$ and $\kappa(G)$ for its vertex-connectivity.
For integers $k\ge1$ and $m\ge2$, let $g(k,m)$ be the least integer $q$ such
that every graph of chromatic number at least $q$ contains a
$(k+1)$-connected subgraph of chromatic number at least $m$.  This line of
work began with Thomassen's 1983 conjecture~\cite{Thomassen1983} and developed
through results of Alon, Kleitman, Thomassen, Saks and Seymour
\cite{AlonEtAl1987}, Chudnovsky, Penev, Scott and Trotignon
\cite{ChudnovskyEtAl2013}, Penev, Thomass\'e and Trotignon
\cite{PenevEtAl2016}, and Gir\~ao and Narayanan~\cite{GiraoNarayanan2022}.
Nguyen proved the following bound.

\begin{theorem}[Nguyen~\cite{Nguyen2024}]\label{thm:nguyen}
For all integers $k\ge1$ and $m\ge2$,
\[
  g(k,m)\le
  \max\left\{m+2k-2,
  \left\lceil\frac{49k}{16}\right\rceil\right\}.
\]
\end{theorem}

Polavarapu sharpened this estimate as follows.

\begin{theorem}[Polavarapu~\cite{Polavarapu2026}]
\label{thm:polavarapu}
For all integers $k\ge1$ and $m\ge2$,
\[
  g(k,m)\le\max\{m+2k-2,3k+1\}.
\]
\end{theorem}

\begin{theorem}[Chv\'atal--Erd\H{o}s~\cite{ChvatalErdos1972}]
\label{thm:chvatal-erdos}
If a graph $H$ has order at least three and
\[
  \kappa(H)\ge\alpha(H),
\]
then $H$ is Hamiltonian.
\end{theorem}

We first bound the defect in terms of the independence number.

\begin{lemma}[An independence-number bound]
\label{lem:alpha-polavarapu}
Every graph $G$ satisfies
\[
  \chi(G)-\psi(G)\le3\alpha(G).
\]
\end{lemma}

\begin{proof}
The assertion is immediate for the empty graph.  Let
\[
  a=\alpha(G),\qquad q=\chi(G),
\]
and assume that $G$ is nonempty.  If $q\le3a$, the claim follows from
$\psi(G)\ge0$.  Assume $q\ge3a+1$, and set
\[
  m=q-2a+2.
\]
Then $m\ge2$, $m+2a-2=q$, and $3a+1\le q$.
Theorem~\ref{thm:polavarapu}, applied with $k=a$, gives an
$(a+1)$-connected subgraph $K$ of $G$ with
\[
  \chi(K)\ge q-2a+2.
\]
Let $H=G[V(K)]$.  Passing to the induced graph can only add edges, and hence
\[
  \kappa(H)\ge a+1,
  \qquad
  \chi(H)\ge q-2a+2,
  \qquad
  \alpha(H)\le a.
\]

If $|V(H)|$ is odd, Theorem~\ref{thm:chvatal-erdos} gives an odd Hamilton
cycle of $H$.  Thus
\[
  \psi(G)\ge\chi(H)\ge q-2a+2.
\]
If $|V(H)|$ is even, remove any vertex $v\in V(H)$.  The standard inequality
$\kappa(H-v)\ge\kappa(H)-1$ gives
\[
  \kappa(H-v)\ge a\ge\alpha(H-v),
\]
and
\[
  \chi(H-v)\ge q-2a+1\ge3.
\]
Theorem~\ref{thm:chvatal-erdos} gives an odd Hamilton cycle of $H-v$, so
\[
  \psi(G)\ge\chi(H-v)\ge q-2a+1.
\]
In both cases,
\[
  \chi(G)-\psi(G)\le2a-1\le3a.
\]
\end{proof}

\begin{proof}[Proof of Theorem~\ref{thm:upper-main}]
Write
\[
  q=\chi(G),\qquad p=\psi(G),\qquad d=q-p.
\]
There is nothing to prove when $d=0$.  Put $G_0=G$.  For each $i\ge0$, as
long as $G_i$ is nonempty, choose a maximum independent set $I_i$ of $G_i$,
write
\[
  a_i=|I_i|=\alpha(G_i),
\]
and set $G_{i+1}=G_i-I_i$.

Deleting an independent set lowers chromatic number by at most one.  Thus,
for $0\le i\le d-1$,
\[
  \chi(G_i)\ge q-i.
\]
Since $G_i$ is an induced subgraph of $G$, we have $\psi(G_i)\le p$.
Lemma~\ref{lem:alpha-polavarapu} gives
\[
  d-i=q-p-i
  \le \chi(G_i)-p
  \le \chi(G_i)-\psi(G_i)
  \le3a_i.
\]
The sets $I_0,\ldots,I_{d-1}$ are pairwise disjoint, so
\[
  N\ge\sum_{i=0}^{d-1}a_i
   \ge\frac13\sum_{i=0}^{d-1}(d-i)
   =\frac{d(d+1)}6.
\]
\end{proof}

\begin{remark}
Using Theorem~\ref{thm:nguyen} in the same proof gives
\[
  \chi(G)-\psi(G)\le\frac{49}{16}\alpha(G)
\]
and
\[
  d(d+1)\le\frac{49}{8}N.
\]
\end{remark}

\section{The palette--code construction}\label{sec:construction}

For a positive integer $p$, write $[p]=\{1,\ldots,p\}$.  For
$x,y\in\{0,1\}^p$, let $\distH(x,y)$ denote their Hamming distance.  For an
integer $r\ge0$, put
\[
  B_p(r)=\sum_{j=0}^{r}\binom pj,
\]
where $\binom pj=0$ for $j>p$.
For a nonempty code $\mathcal C\subseteq\{0,1\}^p$, its covering radius is
\[
  \rho(\mathcal C)=
  \max_{y\in\{0,1\}^p}\min_{x\in\mathcal C}\distH(x,y).
\]
For background on covering radius and Hamming-space coverings, see the survey
of Cohen, Karpovsky, Mattson and Schatz~\cite{CohenEtAl1985} and the monograph
of Cohen, Honkala, Litsyn and Lobstein~\cite{CoveringCodes1997}.

Fix $p\ge2$, a nonempty code $\mathcal C\subseteq\{0,1\}^p$, and an integer
$\ell\ge1$.  We define a graph $G(p,\mathcal C,\ell)$.

First construct a boundary graph $S_p$.  Let
\[
  P=\{a_i^0,a_i^1:i\in[p]\}
\]
induce a clique of order $2p$, and let
\[
  Z=\{z_i:i\in[p]\}
\]
induce a clique of order $p$.  For every $i\in[p]$, join $z_i$ to every
vertex of $P$ except $a_i^0$ and $a_i^1$.  Hence $|V(S_p)|=3p$ and
$\chi(S_p)=2p$: the clique $P$ gives the lower bound, while an upper bound is
obtained by giving $P$ distinct colours and colouring $z_i$ with the colour of
$a_i^0$.  These colours are distinct as $i$ varies, so the colouring is proper
on the clique $Z$.

For each $x=(x_1,\ldots,x_p)\in\mathcal C$, define
\[
  R_x=Z\cup\{a_i^{1-x_i}:i\in[p]\}.
\]
Add a clique $Q_x$ of order $\ell$, make $Q_x$ complete to $R_x$ and
anticomplete to $S_p\setminus R_x$, and make the cliques
$\{Q_x:x\in\mathcal C\}$ pairwise anticomplete.  Figure~\ref{fig:construction}
shows the incidence pattern for one codeword.

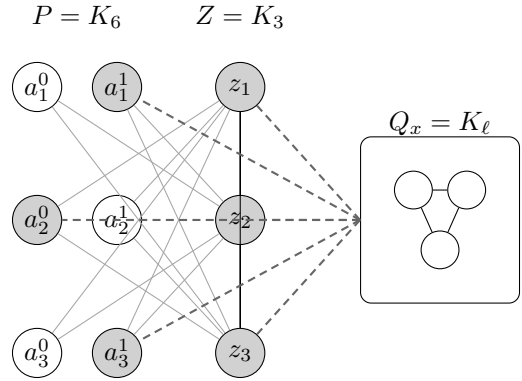
\begin{figure}[H]
\centering
\begin{tikzpicture}[
  scale=.88,
  every node/.style={font=\small},
  v/.style={circle,draw,minimum size=6.5mm,inner sep=0pt},
  sel/.style={v,fill=black!18},
  qbox/.style={draw,rounded corners,minimum width=2.1cm,minimum height=2.2cm},
  jedge/.style={densely dashed,thick,black!60}
]
  \node at (.6,3.05) {$P=K_6$};
  \node[v]   (a10) at (0,2) {$a_1^0$};
  \node[sel] (a11) at (1.2,2) {$a_1^1$};
  \node[sel] (a20) at (0,0) {$a_2^0$};
  \node[v]   (a21) at (1.2,0) {$a_2^1$};
  \node[v]   (a30) at (0,-2) {$a_3^0$};
  \node[sel] (a31) at (1.2,-2) {$a_3^1$};

  \node at (3.05,3.05) {$Z=K_3$};
  \node[sel] (z1) at (3.05,2) {$z_1$};
  \node[sel] (z2) at (3.05,0) {$z_2$};
  \node[sel] (z3) at (3.05,-2) {$z_3$};
  \draw (z1)--(z2)--(z3)--(z1);

  \foreach \u in {a20,a21,a30,a31} {\draw[black!35] (z1)--(\u);}
  \foreach \u in {a10,a11,a30,a31} {\draw[black!35] (z2)--(\u);}
  \foreach \u in {a10,a11,a20,a21} {\draw[black!35] (z3)--(\u);}

  \node[qbox] (q) at (6.05,0) {};
  \node at (6.05,1.45) {$Q_x=K_\ell$};
  \node[v,minimum size=5mm] (q1) at (5.65,.45) {};
  \node[v,minimum size=5mm] (q2) at (6.45,.45) {};
  \node[v,minimum size=5mm] (q3) at (6.05,-.45) {};
  \draw (q1)--(q2)--(q3)--(q1);

  \foreach \u in {a11,a20,a31,z1,z2,z3} {\draw[jedge] (q.west)--(\u);}
\end{tikzpicture}
\caption{A schematic of the palette--code construction for $p=3$ and
$x=(0,1,0)$.  The shaded vertices form $R_x$.  Clique edges inside $P$ are
omitted, and each dashed line represents adjacency from every vertex of
$Q_x$ to the indicated vertex of $R_x$.  The clique $Z$ is drawn explicitly,
and $z_i$ is nonadjacent only to its own coordinate pair in $P$.}
\label{fig:construction}
\end{figure}

A vector $y\in\{0,1\}^p$ defines a \emph{canonical} $2p$-colouring of $S_p$:
give the vertices of $P$ distinct colours and give $z_i$ the colour of
$a_i^{y_i}$.

\begin{observation}[Exact palette count]\label{obs:palette-count}
In the canonical colouring associated with $y$, the set $R_x$ uses exactly
\[
  2p-\distH(x,y)
\]
palette colours.  Equivalently, exactly $\distH(x,y)$ palette colours are
absent from $R_x$.
\end{observation}

\begin{proof}
For coordinate $i$, the selected palette vertex in $R_x$ is
$a_i^{1-x_i}$.  If $x_i=y_i$, then $z_i$ has the colour of $a_i^{x_i}$, so
both coordinate colours occur on $R_x$.  If $x_i\ne y_i$, then
$y_i=1-x_i$, and $z_i$ repeats the colour of $a_i^{1-x_i}$, so exactly one
coordinate colour occurs.  Summing over the coordinates proves the claim.
\end{proof}

\begin{lemma}[Global chromatic number]\label{lem:global-chi}
If $\rho(\mathcal C)<\ell$, then
\[
  \chi(G(p,\mathcal C,\ell))
  =2p+\ell-\rho(\mathcal C).
\]
\end{lemma}

\begin{proof}
Write $G=G(p,\mathcal C,\ell)$ and $\rho=\rho(\mathcal C)$.  Let $\varphi$ be
an arbitrary proper colouring of $G$.  Since $P$ is a clique, its vertices
receive $2p$ distinct colours.  Any palette colour used by $z_i$ must be the
colour of $a_i^0$ or $a_i^1$.  If
$\varphi(z_i)=\varphi(a_i^{y_i})$, choose that value of $y_i$; if $z_i$ uses
a colour outside the palette, choose $y_i$ arbitrarily.  This defines
$y\in\{0,1\}^p$.  Choose $x\in\mathcal C$ with $\distH(x,y)\le\rho$.

For a matching coordinate $x_i=y_i$, the selected vertex
$a_i^{1-x_i}$ and $z_i$ contribute two distinct colours.  For a mismatching
coordinate they contribute at least one colour.  Palette colours from distinct
coordinates are distinct because $P$ is a clique.  If $z_i$ uses a colour
outside the palette, that colour is new, and all such outside colours are
pairwise distinct because $Z$ is a clique.  Hence
\[
  |\varphi(R_x)|\ge
  2\bigl(p-\distH(x,y)\bigr)+\distH(x,y)
  \ge2p-\rho.
\]
Since the clique $Q_x$ is complete to $R_x$, its $\ell$ colours are all
new relative to $R_x$.  Hence
\[
  \chi(G)\ge2p+\ell-\rho.
\]

For the reverse inequality, choose $y\in\{0,1\}^p$ with
\[
  \min_{x\in\mathcal C}\distH(x,y)=\rho
\]
and use the associated canonical colouring of $S_p$.  Introduce one common
pool of $\ell-\rho$ new colours.  For each $x$, first colour
$\min\{\ell,\distH(x,y)\}$ vertices of $Q_x$ with distinct palette colours
absent from $R_x$, whose number is given by Observation~\ref{obs:palette-count}.
Colour the remaining
\[
  \max\{0,\ell-\distH(x,y)\}
\]
vertices with the first colours from the common pool.  Since
$\distH(x,y)\ge\rho$, the pool is large enough.  Distinct cliques $Q_x$ are
anticomplete, so the same pool may be reused.  This gives a
$(2p+\ell-\rho)$-colouring of $G$.
\end{proof}

\begin{lemma}[Cycle-local recolouring]\label{lem:cycle-local}
Suppose that
\[
  \ell\le2p
  \qquad\text{and}\qquad
  3p\,B_p(\ell-1)<2^p.
\]
Then every odd cycle $C$ of $G(p,\mathcal C,\ell)$ satisfies
\[
  \chi(G[V(C)])\le2p.
\]
Moreover, $G(p,\mathcal C,\ell)$ contains an odd cycle whose vertex set spans
a graph of chromatic number $2p$.  Consequently,
\[
  \psi(G(p,\mathcal C,\ell))=2p.
\]
\end{lemma}

\begin{proof}
Write $G=G(p,\mathcal C,\ell)$, and let $C$ be an odd cycle.  If $C$ is
disjoint from $S_p$, then it lies in one clique $Q_x$, because every vertex
outside $S_p$ lies in some $Q_x$ and distinct attached cliques are
anticomplete.  Hence $\chi(G[V(C)])\le\ell\le2p$.

Suppose that $C$ meets $S_p$, and let
\[
  J=\{x\in\mathcal C:V(C)\cap Q_x\ne\varnothing\}.
\]
For each $x\in J$, every component of $C\cap Q_x$ is a path whose two
end-edges join $S_p$, because $Q_x$ has no edges to any other attached clique.
Let $e_C(Q_x,S_p)$ denote the number of cycle edges crossing between $Q_x$ and
$S_p$.  Hence $e_C(Q_x,S_p)\ge2$, including when $C$ visits the same $Q_x$
several times.  Each vertex of $C\cap S_p$ is incident with at most two cycle
edges,
and hence
\[
  2|J|
  \le\sum_{x\in J}e_C(Q_x,S_p)
  \le2|V(C)\cap S_p|.
\]
Thus $|J|\le3p$.
The radius-$(\ell-1)$ balls centred at $J$ do not cover the Hamming cube,
because
\[
  |J|B_p(\ell-1)\le3pB_p(\ell-1)<2^p.
\]
Choose $y\in\{0,1\}^p$ outside their union.  Then
$\distH(x,y)\ge\ell$ for every $x\in J$.  Under the canonical colouring
associated with $y$, Observation~\ref{obs:palette-count} supplies at least
$\ell$ palette colours absent from each $R_x$.  Colour every visited $Q_x$
injectively with such colours.  Since the $Q_x$ are pairwise anticomplete,
this gives a $2p$-colouring of
\[
  G\left[S_p\cup\bigcup_{x\in J}Q_x\right],
\]
and in particular of $G[V(C)]$.

For the reverse inequality, fix $i\in[p]$.  Since $p\ge2$, the vertex $z_i$
has at least two neighbours in the clique $P$.  A complete graph has a
Hamilton path with any prescribed pair of distinct vertices as its ends.
Choose two neighbours of $z_i$ as these ends and add $z_i$.  The resulting cycle has length $2p+1$,
and its vertex set spans a graph containing the clique $P$.  Hence its span
has chromatic number $2p$.
\end{proof}

\begin{corollary}[Exact code-to-defect conversion]\label{cor:code-defect}
Suppose that
\[
  \rho(\mathcal C)<\ell\le2p
  \qquad\text{and}\qquad
  3p\,B_p(\ell-1)<2^p.
\]
Then
\[
  |V(G(p,\mathcal C,\ell))|=3p+\ell|\mathcal C|,
\]
\[
  \chi(G(p,\mathcal C,\ell))=2p+\ell-\rho(\mathcal C),
  \qquad
  \psi(G(p,\mathcal C,\ell))=2p,
\]
and
\[
  \ohdef(G(p,\mathcal C,\ell))=\ell-\rho(\mathcal C).
\]
\end{corollary}

\begin{proof}
The order formula follows from the construction, and the two exact identities
are Lemmas~\ref{lem:global-chi} and~\ref{lem:cycle-local}.
\end{proof}

\section{Covering codes and the polynomial lower bound}
\label{sec:asymptotic}

We need two standard facts about the Hamming cube.  The first is the usual
probabilistic covering estimate, in the spirit of classical nonconstructive
bounds of Cohen~\cite{Cohen1983} and Cohen--Frankl~\cite{CohenFrankl1985}.
The second is a near-middle estimate for Hamming balls, a special case of
standard binomial moderate-deviation asymptotics~\cite{BahadurRao1960}.

\begin{lemma}[Probabilistic covering bound]\label{lem:covering}
For integers $p\ge1$ and $0\le r\le p$, there is a code
$\mathcal C\subseteq\{0,1\}^p$ of covering radius at most $r$ and size
\[
  |\mathcal C|\le
  \left\lceil
  \frac{2^p}{B_p(r)}\bigl(p\log2+1\bigr)
  \right\rceil.
\]
\end{lemma}

\begin{proof}
Choose $M$ points of $\{0,1\}^p$ independently and uniformly, where $M$ is
the integer on the right.  A fixed point is missed by all radius-$r$ balls
with probability
\[
  \left(1-\frac{B_p(r)}{2^p}\right)^M
  \le \exp\left(-\frac{MB_p(r)}{2^p}\right).
\]
Hence the expected number of uncovered points is at most
\[
  2^p\exp\left(-\frac{MB_p(r)}{2^p}\right)<1.
\]
Since this number is a nonnegative integer, some choice leaves no point
uncovered.  Removing repeated centres gives the required code.
\end{proof}

Related polynomial-size coverings in the near-middle regime also appear in
work of Rabani and Shpilka~\cite{RabaniShpilka2010}; see Bazzi~\cite{Bazzi2019}
for further discussion of small codes and covering radius.

\begin{lemma}[Near-middle Hamming balls]\label{lem:moderate}
For every fixed $c>0$,
\[
  2^{-p}B_p\left(
  \left\lfloor\frac p2-c\sqrt{p\log p}\right\rfloor
  \right)
  =p^{-2c^2+o(1)}
\]
as $p\to\infty$.  The same estimate holds if the displayed radius is
changed by $O(1)$.
\end{lemma}

\begin{proof}
Put
\[
  r=\left\lfloor\frac p2-c\sqrt{p\log p}\right\rfloor,
  \qquad
  t=\frac p2-r=c\sqrt{p\log p}+O(1).
\]
If $X\sim\operatorname{Bin}(p,1/2)$, the Chernoff bound gives
\[
  2^{-p}B_p(r)
  =\Pr(X\le p/2-t)
  \le \exp\left(-\frac{2t^2}{p}\right)
  =p^{-2c^2+o(1)}.
\]
For the lower bound, Stirling's formula in entropy form yields
\[
  2^{-p}\binom pr
  =\frac{1+o(1)}{\sqrt{\pi p/2}}
    \exp\left(-\frac{2t^2}{p}
    +O\left(\frac{t^4}{p^3}\right)\right)
  =p^{-1/2-2c^2+o(1)}.
\]
For $0\le j\le p/(8t)$,
\[
  \frac{\binom p{r-j}}{\binom pr}
  =\prod_{s=0}^{j-1}\frac{r-s}{p-r+s+1}
  \ge
  \exp\left[-O\left(\frac{tj}{p}+\frac{j^2}{p}+\frac jp\right)\right].
\]
Indeed, taking logarithms and using
$\log(1-u)\ge-u-O(u^2)$ gives the displayed estimate uniformly in this
range.  Here $tj/p\le1/8$, $j^2/p=O(1/\log p)$, and $j/p=o(1)$, so the ratio
is bounded below by a positive constant independent of $p$ and $j$.  The
number of terms is
\[
  \left\lfloor\frac{p}{8t}\right\rfloor+1
  =\Theta\left(\sqrt{\frac p{\log p}}\right)
  =p^{1/2-o(1)}.
\]
Summing these terms below $r$ gives
\[
  2^{-p}B_p(r)\ge p^{-2c^2+o(1)}.
\]
Changing $r$ by $O(1)$ changes the estimate only by a factor $p^{o(1)}$.
\end{proof}

We now choose the two radii needed in Corollary~\ref{cor:code-defect}.

\begin{theorem}[Polynomial seed family]\label{thm:seed-family}
Fix real numbers
\[
  \frac1{\sqrt2}<a<b.
\]
For every $\varepsilon$ with $0<\varepsilon<b-a$ and all sufficiently large
integers $p$, there is a graph $G_p$ satisfying
\[
  |V(G_p)|\le p^{2+2b^2+\varepsilon}
\]
and
\[
  \ohdef(G_p)\ge
  (b-a-\varepsilon)\sqrt{p\log p}.
\]
\end{theorem}

\begin{proof}
Fix $\varepsilon>0$, and set
\[
  \ell_p=
  \left\lfloor\frac p2-a\sqrt{p\log p}\right\rfloor,
  \qquad
  r_p=
  \left\lfloor\frac p2-b\sqrt{p\log p}\right\rfloor.
\]
For all sufficiently large $p$,
\[
  0\le r_p<\ell_p\le2p.
\]
Lemma~\ref{lem:moderate} gives
\[
  \frac{B_p(\ell_p-1)}{2^p}=p^{-2a^2+o(1)}.
\]
Since $2a^2>1$,
\[
  3pB_p(\ell_p-1)<2^p
\]
for all sufficiently large $p$.

By Lemmas~\ref{lem:covering} and~\ref{lem:moderate}, there is a code
$\mathcal C_p\subseteq\{0,1\}^p$ with
\[
  \rho(\mathcal C_p)\le r_p.
\]
The ceiling in Lemma~\ref{lem:covering} contributes at most one, and the main
term tends to infinity.  Hence, for all sufficiently large $p$,
\[
  |\mathcal C_p|
  \le p^{1+2b^2+\varepsilon/2}.
\]
In particular,
\[
  \rho(\mathcal C_p)\le r_p<\ell_p,
\]
so all hypotheses of Corollary~\ref{cor:code-defect} are satisfied for
\[
  G_p=G(p,\mathcal C_p,\ell_p).
\]
It follows that
\[
  \ohdef(G_p)
  =\ell_p-\rho(\mathcal C_p)
  \ge\ell_p-r_p.
\]
The floor functions cost at most one in this difference, so
\[
  \ell_p-r_p\ge(b-a)\sqrt{p\log p}-1
  \ge(b-a-\varepsilon)\sqrt{p\log p}
\]
for all sufficiently large $p$.  Moreover, since $\ell_p\le p$,
\[
  |V(G_p)|
  =3p+\ell_p|\mathcal C_p|
  \le p^{2+2b^2+\varepsilon}
\]
for all sufficiently large $p$.
\end{proof}

\begin{remark}[Relative scale of the construction]
The codes in Theorem~\ref{thm:seed-family} have polynomial size.  The
sphere-covering inequality
\[
  2^p\le |\mathcal C_p|B_p(\rho(\mathcal C_p))
\]
and the Chernoff upper bound imply
\[
  \rho(\mathcal C_p)\ge\frac p2-O(\sqrt{p\log p}).
\]
Since also $\ell_p=p/2-O(\sqrt{p\log p})$, the exact conversion gives
\[
  \ohdef(G_p)=O(\sqrt{p\log p})=o(p).
\]
As $\psi(G_p)=2p$, it follows that $\chi(G_p)/\psi(G_p)\to1$.  The
construction gives a polynomially large additive gap within this asymptotically
unit ratio.
\end{remark}

\begin{proof}[Proof of Theorem~\ref{thm:lower-main}]
Fix $\delta>0$.  Choose $b>1/\sqrt2$ and $\eta>0$ so small that
\[
  \frac{1}{2(2+2b^2+\eta)}>\frac16-\frac\delta2,
\]
and then choose $a$ with $1/\sqrt2<a<b$.  Put
\[
  \Gamma=2+2b^2+\eta.
\]
Choose
\[
  0<\varepsilon_0<\min\left\{\eta,\frac{b-a}{2}\right\}
\]
and apply Theorem~\ref{thm:seed-family} with error parameter
$\varepsilon_0$.  Then, for all sufficiently large $p$, there is a graph
$G_p$ with
\[
  |V(G_p)|\le p^\Gamma
\]
and
\[
  \ohdef(G_p)\ge c\sqrt{p\log p},
  \qquad
  c=b-a-\varepsilon_0>0.
\]

Given sufficiently large $N$, take
\[
  p=\left\lfloor N^{1/\Gamma}\right\rfloor.
\]
Then $|V(G_p)|\le N$.  Moreover, for all sufficiently large $N$,
\[
  c\sqrt{p\log p}
  \ge N^{1/(2\Gamma)-\delta/2}
  \ge N^{1/6-\delta}.
\]
Hence $F(N)\ge N^{1/6-\delta}$.
\end{proof}

The same construction may also be expressed in the following linear-seed form: for every
$\varepsilon>0$, there are arbitrarily large integers $t$ and graphs $G$ with
\[
  \ohdef(G)\ge t,
  \qquad
  |V(G)|\le t^{6+\varepsilon}.
\]
To see this, choose $a$ and $b$ sufficiently close to $1/\sqrt2$, apply
Theorem~\ref{thm:seed-family} with a sufficiently small error parameter, and
set $t=\lfloor c\sqrt{p\log p}\rfloor$ for the resulting constant $c>0$.

\section{Exponent transfer and discrepancy}

The palette--code construction converts the size and covering radius of a code
into a lower-bound exponent for $F$.  The following proposition records this
transfer.

\begin{proposition}[Exponent transfer]\label{prop:conditional-transfer}
Let $a_p,b_p$ be sequences satisfying
\[
  \frac1{\sqrt2}<a_p<b_p,
  \qquad
  \log\frac1{b_p-a_p}=o(\log p),
\]
and put
\[
  \ell_p=
  \left\lfloor\frac p2-a_p\sqrt{p\log p}\right\rfloor.
\]
Assume that
\[
  3pB_p(\ell_p-1)<2^p
\]
and that there are codes $\mathcal C_p\subseteq\{0,1\}^p$ such that
\[
  \rho(\mathcal C_p)
  \le
  \left\lfloor\frac p2-b_p\sqrt{p\log p}\right\rfloor,
  \qquad
  |\mathcal C_p|=p^{\beta+o(1)}
\]
for some constant $\beta\ge0$.  For
\[
  N_p=|V(G(p,\mathcal C_p,\ell_p))|,
\]
one has
\[
  F(N_p)\ge N_p^{1/(2+2\beta)-o(1)}.
\]
\end{proposition}

\begin{proof}
The hypotheses imply $\rho(\mathcal C_p)<\ell_p$ for all sufficiently large
$p$, so Corollary~\ref{cor:code-defect} applies.  Since
$\log(1/(b_p-a_p))=o(\log p)$,
\[
  \ohdef(G(p,\mathcal C_p,\ell_p))
  \ge (b_p-a_p)\sqrt{p\log p}-1
  =p^{1/2-o(1)}.
\]
Moreover, $\ell_p=\Theta(p)$ and hence
\[
  |V(G(p,\mathcal C_p,\ell_p))|
  =3p+\ell_p|\mathcal C_p|
  =p^{1+\beta+o(1)}.
\]
Eliminating $p$ gives the stated exponent.
\end{proof}

For the random covering used in Section~\ref{sec:asymptotic}, the code-size
exponent approaches $\beta=2$ as the two coefficients approach $1/\sqrt2$,
recovering the exponent $1/6$.  The sphere-covering inequality suggests
$\beta=1$, corresponding to exponent $1/4$.  Spencer's discrepancy theorem
forces the code-size exponent above $1$ in the near-critical regime.

\begin{theorem}[Spencer's set-system discrepancy bound]
\label{thm:spencer}
There is an absolute constant $K>0$ such that, for every family
$A_1,\ldots,A_m\subseteq[n]$, there is a sign vector
$\varepsilon\in\{\pm1\}^n$ satisfying
\[
  \max_{j\in[m]}
  \left|\sum_{i\in A_j}\varepsilon_i\right|
  \le K\sqrt{n\log\left(\frac mn+2\right)}.
\]
\end{theorem}

This standard form follows from Spencer's set-system discrepancy theorem;
see Spencer~\cite{Spencer1985} and Matou\v{s}ek~\cite[Theorem~4.2]{Matousek1999}.
For $m<n$, adjoining empty sets reduces the statement to the case $m=n$.

\begin{proposition}[Near-critical code-size lower bound]\label{prop:discrepancy-barrier}
There is an absolute constant $\gamma>0$ such that the following holds.  If
$b_p=1/\sqrt2+o(1)$ and a code
$\mathcal C_p\subseteq\{0,1\}^p$ has covering radius at most
\[
  \left\lfloor\frac p2-b_p\sqrt{p\log p}\right\rfloor,
\]
then, for all sufficiently large $p$,
\[
  |\mathcal C_p|\ge p^{1+\gamma}.
\]
\end{proposition}

\begin{proof}
Let $M=|\mathcal C_p|$ and, for each $x\in\mathcal C_p$, put
\[
  S_x=\{i\in[p]:(-1)^{x_i}=1\}.
\]
Apply Theorem~\ref{thm:spencer} to the $M+1$ sets
\[
  \{S_x:x\in\mathcal C_p\}\cup\{[p]\}.
\]
There is $\varepsilon\in\{\pm1\}^p$ such that, with
\[
  L=\log\left(\frac{M+1}{p}+2\right),
\]
one has
\[
  \left|\sum_{i\in S_x}\varepsilon_i\right|
  \le K\sqrt{pL}
  \quad\text{for every }x\in\mathcal C_p,
  \qquad
  \left|\sum_{i=1}^p\varepsilon_i\right|
  \le K\sqrt{pL}.
\]
Let $\sigma(x)_i=(-1)^{x_i}$.  Since
\[
  \langle\sigma(x),\varepsilon\rangle
  =2\sum_{i\in S_x}\varepsilon_i-
    \sum_{i=1}^p\varepsilon_i,
\]
it follows that
\[
  |\langle\sigma(x),\varepsilon\rangle|
  \le3K\sqrt{pL}
  \qquad\text{for every }x\in\mathcal C_p.
\]
Choose $y\in\{0,1\}^p$ with $\sigma(y)=\varepsilon$.  By the covering
property, some $x\in\mathcal C_p$ satisfies
\[
  \distH(x,y)\le
  \frac p2-b_p\sqrt{p\log p}.
\]
For this $x$,
\[
  \langle\sigma(x),\varepsilon\rangle
  =p-2\distH(x,y)
  \ge(\sqrt2-o(1))\sqrt{p\log p}.
\]
Comparing the last two displays gives
\[
  L\ge\left(\frac{2}{9K^2}-o(1)\right)\log p.
\]
For all sufficiently large $p$,
\[
  L\ge\frac{1}{9K^2}\log p.
\]
It follows that
\[
  M+1\ge p\left(p^{1/(9K^2)}-2\right).
\]
For all sufficiently large $p$,
\[
  p^{1/(9K^2)}-2\ge2p^{1/(18K^2)}.
\]
Taking $\gamma=1/(18K^2)$ gives $M\ge p^{1+\gamma}$ for all sufficiently
large $p$.
\end{proof}

Combining Propositions~\ref{prop:conditional-transfer} and
\ref{prop:discrepancy-barrier}, every near-critical implementation satisfying
the hypotheses of Proposition~\ref{prop:conditional-transfer} has
$\beta\ge1+\gamma$.  The lower-bound exponent furnished by
Proposition~\ref{prop:conditional-transfer} is at most
$1/(4+2\gamma)<1/4$.

\begin{question}
Is
\[
  F(N)=N^{1/2-o(1)}?
\]
More strongly, is $F(N)=\Theta(\sqrt N)$?
\end{question}

\end{document}